\documentclass[12pt]{article} \textheight 8.8 true in \textwidth
6.2 true in

\usepackage[T2A]{fontenc}
\usepackage[cp1251]{inputenc}
\usepackage[english]{babel}
\usepackage{amsfonts}
\usepackage{amssymb,amsmath,amstext}
\usepackage{amscd, amsthm, latexsym}
\usepackage{color}
\usepackage[final]{graphicx}
\usepackage{epsfig}
\usepackage{euscript}
\usepackage{srcltx,epsf}
\usepackage{wrapfig}
\usepackage{mathrsfs}
\usepackage{ifthen}
\usepackage{cleveref}
\usepackage{multirow}
\usepackage{array}
\usepackage{slashbox}

\def\hide#1{}
\def\old#1{}

\def\oop#1{}
\def\gap#1{}

\theoremstyle{plain}
\newtheorem{theorem}{Theorem}
\newtheorem{proposition}{Proposition}
\newtheorem*{theorem*}{Theorem}
\newtheorem{lemma}{Lemma}
\newtheorem{remark}{Remark}
\newtheorem{corollary}{Corollary}
\newtheorem{definition}{Definition}

\def \ZZ  {\Bbb Z}

\begin{document}

\newcounter{figcounter}
\setcounter{figcounter}{0} \addtocounter{figcounter}{1}

\title{The enumeration of coverings of closed orientable Euclidean manifolds $\mathcal{G}_{3}$ and $\mathcal{G}_{5}$.}
\author{
G.~Chelnokov\thanks{This work was supported by the Russian Foundation for Basic Research (grant $18-01-00036\slash18$).}\\
{\small\em National Research University Higher School of Economics} \\
{\small\tt grishabenruven@yandex.ru }
\\ [2ex]
A.~Mednykh\thanks{This work was supported by the Russian Foundation for Basic Research (grant 16-31-00138).}\\
{\small\em Sobolev Institute of Mathematics,} \\
{ \small\em Novosibirsk, Russia}\\
{\small\em Novosibirsk State University,}\\
{\small\em Novosibirsk, Russia}\\
{\small\tt mednykh@math.nsc.ru} }
\date{}
\maketitle \maketitle

\begin{abstract}
There are only 10 Euclidean forms, that is  flat closed three
dimensional manifolds: six are orientable and four are
non-orientable. The aim of this paper is to describe all types of
$n$-fold coverings over orientable Euclidean manifolds
$\mathcal{G}_{3}$ and $\mathcal{G}_{5}$, and calculate the numbers
of non-equivalent coverings of each type. We classify subgroups in
the fundamental groups $\pi_1(\mathcal{G}_{3})$ and
$\pi_1(\mathcal{G}_{5})$ up to isomorphism and calculate the numbers
of conjugated classes of each type of subgroups for index $n$. The
manifolds $\mathcal{G}_{3}$ and $\mathcal{G}_{5}$ are uniquely
determined among the others orientable forms by their homology
groups $H_1(\mathcal{G}_{3})=\ZZ_3\times  \ZZ$ and
$H_1(\mathcal{G}_{5})= \ZZ$.


\end{abstract}

\section*{Introduction}
Let  $\mathcal{M}$  be a manifold with fundamental group
$\Gamma=\pi_{1}(\mathcal{M}).$ Two coverings $$p_1: \mathcal{M}_1
\to \mathcal{M}   \text {      and       } p_2: \mathcal{M}_2 \to
\mathcal{M}  $$ are said to be equivalent if there exists a
homeomorphism $h:  \mathcal{M}_1 \to \mathcal{M}_2$ such that $p_1
=p_2 \circ h.$ According to the general theory of covering spaces,
any $n$-fold covering is uniquely determined by a subgroup of index
$n$ in the group $\Gamma.$ The equivalence classes of $n$-fold
covering of $\mathcal{M}$ are in one-to-one correspondence with the
conjugacy
 classes of subgroups of index $n$ in the fundamental group
 $\pi_1(\mathcal{M}).$ See,
for example, (\cite{Hatcher}, p.~67).
  In such a way the following two natural problems arise. The first one is to calculate the number
 of subgroups of given finite index $n$ in $\pi_1(\mathcal{M}).$  The second problem is to find the number of conjugacy
 classes of subgroups of index $n$ in $\pi_1(\mathcal{M}).$

The problem     of enumeration for nonequivalent coverings over a
Riemann surface with given branch type goes back to the paper
\cite{Hur91} by Hurwitz, in which the number of coverings over the
Riemann sphere with given number of simple (of order two) branching
points was determined. Later, in \cite{Hur02}, it has been proved
that this number can be expressed  in the terms of irreducible
characters of symmetric groups. The Hurwitz problem was studed by
many authors. A detailed survey of the related results is contained
in  (\cite{Lub}, \cite{KwakLeeMed}). For closed Riemann surfaces,
this problem was completely solved in \cite{MedHurwitz}. However, of
most interest is the case of unramified coverings. Let $s_\Gamma(n)$
denote the number of subgroups of index $n$ in the group $\Gamma,$
and let $c_\Gamma(n)$ be the number of conjugacy classes of such
subgroups. According to what was said above, $c_\Gamma(n)$
coincides with the number of nonequivalent $n$-fold coverings over a
manifold $\mathcal{M}$  with fundamental group $\Gamma.$
The numbers $s_\Gamma(n)$ and $c_\Gamma(n)$  for the fundamental
group of a closed surface (orientable or not) were found  in
(\cite{Med78}, \cite{Med79}, \cite{MP86}). In the paper \cite{Medn},
a general method for calculating the number $c_\Gamma(n)$  of
conjugacy classes of subgroups in an arbitrary finitely generated
group $\Gamma$  was given.  Asymptotic formulas for $s_\Gamma(n)$ in
many important cases were obtained by T.~W.~M\"uller and his
collaborators (\cite{Mul},  \cite{MulShar}, \cite{MulPuch}).

In the three-dimensional case, for a large class of Seifert
fibrations, the value of $s_\Gamma(n)$ was calculated  in
\cite{LisMed00} and \cite{LisMed12}. In our previous papers
\cite{We1} and \cite{We2} the numbers $s_\Gamma(n)$ and
$c_\Gamma(n)$  were determined for the fundamental group of
non-orientable Euclidian manifolds   $\mathcal{B}_1$ and
$\mathcal{B}_2$ whose   homologies  are ${H}_1(\mathcal{B}_1)=
\mathbb{Z}_2\oplus\mathbb{Z}^2$ and ${H}_1(\mathcal{B}_2)=
\mathbb{Z}^2$ and for the orientable Euclidean manifolds
$\mathcal{G}_2$ and $\mathcal{G}_4$ with   ${H}_1(\mathcal{G}_2)=
\mathbb{Z}_2\oplus\mathbb{Z}_2\oplus\mathbb{Z}$ and
${H}_1(\mathcal{G}_4)= \mathbb{Z}_2\oplus\mathbb{Z}$ respectively.

 The aim of the present paper is to investigate
$n$-fold coverings over orientable Euclidean three dimensional
manifolds $\mathcal{G}_{3}$ and $\mathcal{G}_{5}$, with the first
homologies
  $H_1(\mathcal{G}_{3})=\ZZ_3\oplus \ZZ$ and
$H_1(\mathcal{G}_{5})=\ZZ$.  We classify subgroups of finite index
in the fundamental groups of $\pi_1(\mathcal{G}_{3})$ and
$\pi_1(\mathcal{G}_{5})$ up to isomorphism. Then we calculate the
number of subgroups and the number of conjugacy classes of subgroups
of each isomorphism type for a given index $n$.

 We note that numerical methods to solve these and similar problems for the three-dimensional crystallogical  groups were
developed by the Bilbao group \cite{babaika}. The first homologies
of all the three-dimensional crystallogical  groups are determined
in \cite{Ratc}.

\subsection*{Notations}
We  use the following notations: $s_{H,G}(n)$ is the number of
subgroups of index $n$ in the group $G$, isomorphic to the group
$H$; $c_{H,G}(n)$ is the number conjugacy classes of subgroups of
index $n$ in the group $G$, isomorphic to the group $H$. 

Also we will need the following combinatorial functions:
$$\sigma_0(n) = \sum_{k \mid n}1,\quad \sigma_1(n) =  \sum_{k \mid n} k,\quad \sigma_2(n) =  \sum_{k \mid n} \sigma_1(k),\quad \omega(n) = \sum_{k \mid n} k\sigma_1(k),$$
$$\theta(n)=|\{(p,q)|p>0, q\ge 0, p^2-pq+q^2=n\}|.$$
In all cases we consider the function vanished if $n\not\in
\Bbb{N}$.


{\bf Remark.} It can be shown that $\theta(n)=\sum_{k \mid
n}\big(\frac{k}{3}\big)$, where $\big(\frac{k}{3}\big)$ is the
Legendre symbol, see \cite{Conw-Sloan} p.112. This representation
clarifies the analogy between the functions $\sigma_1(n)$ and
$\theta(n)$, and makes the appearance of the latter one less
amazing. However, this representation will not be used further.

\section{Overview}
The main goal of this paper is to prove the following results.

The first theorem provides the complete solution of the problem of
enumeration of subgroups of a given finite index in
$\pi_{1}(\mathcal{G}_{3})$.
\begin{theorem}\label{th-1-tricosm}
Every subgroup $\Delta$ of finite index $n$ in
$\pi_{1}(\mathcal{G}_{3})$ is isomorphic to either
$\pi_{1}(\mathcal{G}_{3})$ or $\pi_1(\mathcal{G}_{1})\cong\ZZ^3$.
The respective numbers of subgroups are
$$
s_{\pi_1(\mathcal{G}_{3}), \pi_1(\mathcal{G}_{3})}(n) =  \sum_{k
\mid n}k\theta(k)-\sum_{k \mid \frac{n}{3}}k\theta(k), \leqno (i)
$$
$$
s_{\ZZ^3, \pi_1(\mathcal{G}_{3})}(n) =  \omega(\frac{n}{3}).\leqno
(ii)
$$
\end{theorem}

The next theorem provides the number of conjugacy classes of
subgroups of index $n$ in $\pi_{1}(\mathcal{G}_{3})$ for each
isomorphism type. That is the number of non-equivalent $n$-fold
covering $\mathcal{G}_{3}$, which have a prescribe fundamental
group.

\begin{theorem}\label{th-2-tricosm}
Let $\mathcal{N} \to \mathcal{G}_{3}$ be an $n$-fold covering over
$\mathcal{G}_{3}$. If $n$ is not divisible by 3 then $\mathcal{N}$
is homeomorphic to $\mathcal{G}_{3}$. If $n$ is divisible by 3 then
$\mathcal{N}$ is homeomorphic to either $\mathcal{G}_{3}$ or
$\mathcal{G}_{1}$. The corresponding numbers of nonequivalent
coverings are given by the following formulas:
$$
c_{\pi_{1}(\mathcal{G}_{3}),\pi_{1}(\mathcal{G}_{3})}(n)
=\sum_{k\mid n} \theta(k)+\sum_{k\mid \frac{n}{3}}
\theta(k)-2\sum_{k\mid \frac{n}{9}} \theta(k) \leqno (i)
$$
$$
c_{\ZZ^3,\pi_{1}(\mathcal{G}_{3})}(n)
=\frac{1}{3}\Big(\omega(\frac{n}{3})+2\sum_{k\mid
\frac{n}{3}}\theta(k)+4\sum_{k\mid \frac{n}{9}}\theta(k)\Big).
\leqno (ii)
$$
\end{theorem}

The next two  theorems are analogues of \Cref{th-1-tricosm} and
\Cref{th-2-tricosm} respectively for the manifold $\mathcal{G}_5$.

\begin{theorem}\label{th-1-hexacosm}Every subgroup $\Delta$ of finite index $n$ in
$\pi_{1}(\mathcal{G}_{5})$ is isomorphic to either
$\pi_{1}(\mathcal{G}_{5})$ or $\pi_{1}(\mathcal{G}_{3})$ or
$\pi_{1}(\mathcal{G}_{2})$ or $\pi_1(\mathcal{G}_{1})\cong\ZZ^3$.
The respective numbers of subgroups are
$$
s_{\pi_1(\mathcal{G}_{5}), \pi_1(\mathcal{G}_{5})}(n) =  \sum_{k
\mid n,\, (\frac{n}{k},6)=1}k\theta(k) \leqno (i)
$$
$$
s_{\pi_1(\mathcal{G}_{3}), \pi_1(\mathcal{G}_{5})}(n)=\sum_{k \mid
\frac{n}{2}}k\theta(k)-\sum_{k \mid \frac{n}{6}}k\theta(k) \leqno
(ii)
$$
$$
s_{\pi_1(\mathcal{G}_{2}),
\pi_1(\mathcal{G}_{5})}(n)=\omega(\frac{n}{3})-\omega(\frac{n}{6})
\leqno (iii)
$$
$$
s_{\ZZ^3, \pi_1(\mathcal{G}_{5})}(n)=\omega(\frac{n}{6}). \leqno
(iv)
$$

\end{theorem}

\begin{theorem}\label{th-2-hexacosm}
The numbers of $n$-fold covering over $\mathcal{G}_{5}$  is given by
the following formulas:
$$
c_{\pi_{1}(\mathcal{G}_{5}),\pi_{1}(\mathcal{G}_{5})}(n) =
 \sum_{k\mid n,\, (\frac{n}{k},6)=1}\theta(k) \leqno (i)
$$
$$
c_{\pi_{1}(\mathcal{G}_{3}),\pi_{1}(\mathcal{G}_{5})}(n) =
\sum_{k\mid\frac{n}{2}}\theta(k)-\sum_{k\mid\frac{n}{18}}\theta(k)
\leqno (ii)
$$
$$
c_{\pi_{1}(\mathcal{G}_{2}),\pi_{1}(\mathcal{G}_{5})}(n)=\frac{1}{3}\Big(\sigma_2(\frac{n}{3})+2\sigma_2(\frac{n}{6})-3\sigma_2(\frac{n}{12})+2\sum_{k\mid
\frac{n}{3}}\theta(k)-2\sum_{k\mid \frac{n}{6}}\theta(k)\Big) \leqno
(iii)
$$
$$
c_{\ZZ^3,\pi_{1}(\mathcal{G}_{5})}(n)=\frac{1}{6}\Big(\omega(\frac{n}{6})+\sigma_2(\frac{n}{6})+3\sigma_2(\frac{n}{12})
+4\sum_{k\mid\frac{n}{6}}\theta(k)+4\sum_{k\mid\frac{n}{18}}\theta(k)\Big)\leqno
(iv).
$$
\end{theorem}
In the Appendix we present the Dirichlet generating functions for
the above sequences.

\section{Preliminaries}
Further we use the representations for the fundamental groups
$\pi(\mathcal{G}_3)$ and $\pi(\mathcal{G}_5)$ given in \cite{Wolf}
and \cite{Conway}.

\begin{equation}\label{fund_G3}
 \pi_{1}(\mathcal{G}_{3})=\langle x, y, z: xyx^{-1}y^{-1}=1,
   zxz^{-1}= y, zyz^{-1}=(xy)^{-1}
 \rangle .
 \end{equation}
\begin{equation}\label{fund_G5}
 \pi_{1}(\mathcal{G}_{5})=\langle \tilde{x}, \tilde{y}, \tilde{z}: \tilde{x}\tilde{y}\tilde{x}^{-1}\tilde{y}^{-1}=1,
   \tilde{z}\tilde{x}\tilde{z}^{-1}= \tilde{x}\tilde{y}, \tilde{z}\tilde{y}\tilde{z}^{-1}=\tilde{x}^{-1}
 \rangle .
 \end{equation}

We will widely use the following statements.

\begin{proposition}\label{number of sublattices}
Let $\Delta$ be a subgroup of finite index $n$ in $\ZZ^2$. Then
$\Delta$ have a pair of generators of the form $(a,0)$ and $(\mu,b)$
where $a$ and $b$ are positive integers with $ab=n$ and $\mu$ is a
nonnegative integer with $0 \le \mu < a$. Furthermore, the set of
subgroups $\Delta$ with $|\ZZ^2:\Delta|=n$ bijectively corresponds
to the set of pairs of generators of described form. The number of
such subgroups $\Delta$ is $\sigma_1(n)$.

Let $\Delta$ be a subgroup of finite index $n$ in $\ZZ^3$. Then
$\Delta$ have a set of three generators $(a,0,0)$, $(\mu,b,0)$ and
$(\nu,\lambda, c)$ where $a$, $b$, $c$ are positive integers with
$abc=n$, $\mu$, $\nu$ are integers with $0 \le \mu,\nu < a$ and
$\lambda$ is an integer with $0 \le \lambda < b$ . Furthermore, the
set of subgroups $\Delta$ with $|\ZZ^3:\Delta|=n$ bijectively
corresponds to the set of triplets of generators of described form.
The number of such subgroups $\Delta$ is $\omega(n)$.
\end{proposition}

\begin{corollary}\label{number of halfes}
Given an integer $n$, by $S(n)$ denote the number of pairs
$(H,\nu)$, where $H$ is a subgroup of index $n$ in $\ZZ^2$ and $\nu$
is a coset of $\ZZ^2/H$ with $2\nu=0$ (we use the additive
notation). Then
$$
S(n)=\sigma_1(n)+3\sigma_1(\frac{n}{2}).
$$
\end{corollary}

\begin{lemma}\label{index_and_ker}
Let $H\leqslant G$ be an abelian group and its subgroup of finite
index. Let $\phi: G \to G$ be an endomorphism of $G$, such that
$\phi(H)\leqslant H$ and the index $|G:\phi(G)|$ is also finite.
Then the cardinality of kernel of $\phi: G/H \to G/H$ equals to the
index $|G : (H+\phi(G))|$.
\end{lemma}
%
For the proofs, see Lemma 1 in \cite{We2}

\begin{remark}\label{number of halfes-remark}
Combining \Cref{index_and_ker} and \Cref{number of halfes} we get
the following observation. Given a subgroup $H\leqslant \ZZ^2$, the
number of $\nu \in \ZZ^2/H$, such that $2\nu=0$, is equal to
$|\ZZ^2/\langle (2,0), (0,2), H\rangle|$. Indeed, taking $\ZZ^2$ as
$G$, $H$ as $H$ and $\phi: g\to 2g,\, g\in \ZZ^2$ as $\phi$ one gets
the desired equality. Since for each $H$ the numbers $|\{\nu| \nu\in
\ZZ^2/H,\,2\nu=0\}|$ and $|\ZZ^2/\langle (2,0), (0,2), H\rangle|$
coincide, their sums taken over all subgroups $H$ also coincide,
that is
$$
S(n)=\sum_{H\leqslant \ZZ^2,\,|\ZZ^2/H|=n}|\{\nu| \nu\in
\ZZ^2/H,\,2\nu=0\}|=\sum_{H\leqslant \ZZ^2,\,|\ZZ^2/H|=n}
|\ZZ^2/\langle (2,0), (0,2), H\rangle|.
$$
\end{remark}

\begin{definition}\label{ell}
Consider the group $\ZZ^2$. By $\ell$ denote the automorphism $\ell:
\ZZ^2 \to \ZZ^2$ given by $(x,y) \to (-y,x-y)$.
\end{definition}

\begin{lemma}\label{lemG3-3.1}\label{lemG3-3.2}
A subgroup $H \leqslant \ZZ^2$ is preserved by $\ell$ if and only if
$H$ is generated by a pair of elements of the form $(p,q),(-q,p-q)$.
In this case $|\ZZ^2/H|=p^2-pq+q^2$. For a given integer $n$ the
number of invariant under $\ell$  subgroups $H$ of index $n$ in
$\ZZ^2$ is $\theta(n)$.
\end{lemma}
\begin{proof}
Suppose $H$ is generated by elements $(p,q)$ and
$(-q,p-p)=\ell\big((p,q)\big)$. Then obviously $\ell(H)=H$. Also
$|\ZZ^2/H|=p^2-pq+q^2$, since $p^2-pq+q^2$ is the number of integer
points in a fundamental domain of $H$ on the plane.

Vice versa, suppose $\ell(H)=H$. Denote $d(x,y)=x^2-xy+y^2$. Let
$u=(p,q)\in H\setminus\{0\}$ be an element with the minimal value of
$d(u)$. Consider the subgroup $H_1=\langle u,\ell(u) \rangle
\leqslant H$. Assume $H_1 \neq H$ and $v \in H\setminus H_1$. Since
$H_1=\langle u,\ell(u) \rangle$, the fundamental domain of $H_1$ is
a parallelogram with vertices $0,u,\ell(u),u+\ell(u)$. That means
that the plane splits into the parallelograms of the form
$w,w+u,w+\ell(u),w+u+\ell(u), w \in H_1$, each of them splits into
two right triangles. One of this triangles contains $v$. Note that
the distance from a point inside a right triangle to one of its
vertices is not greater then the side of this triangle. This
contradicts the minimality of $d(u)$, thus $H_1=H$.

To find the number of subgroups $H$ note that the number of pairs
$(p,q)$ with $p^2-pq+q^2=n,\,p>0,\,q\ge0$ is $\theta(n)$. As it was
proven above, for each pair $(p,q)$ of the above type two pairs
$(p,q)$ and $\ell\big((p,q)\big)$ generate a subgroup $H$ of the
required type. Moreover $d(p,q)$ takes the minimal value among
$d(v),\,v\in H\setminus\{0\}$. Suppose two different pairs $(p,q)$
and $(p',q')$ correspond the same subgroup $H$. Then $(p-p',q-q')\in
H$ and $0<d(p-p',q-q')<d(p,q)$, which contradiction proves that
there is a one-to-one correspondence between pairs $(p,q)$ and
subgroups $H$.
\end{proof}

Before formulating the next corollary note that $\ell(\nu), \nu \in
\ZZ^2/H$ is well-defined if $\ell(H) \le H$.

\begin{corollary}\label{number of thirds}
Let $n$ be an integer. Consider the set of all subgroups $H$ of
Given an integer $n$, by $T(n)$ denote the number of pairs
$(H,\nu)$, where $H$ is a subgroup of index $n$ in $\ZZ^2$ with
$\ell(H)=H$ and $\nu$ is a coset of $\ZZ^2/H$ with $\ell(\nu)=\nu$.
Then
$$
T(n)=\theta(n)+2\theta(\frac{n}{3}).
$$
\end{corollary}

\begin{proof}
Consider subgroup $H \leqslant \ZZ^2$ with $|\ZZ^2:H|=n$ and
$\ell(H)=H$. \Cref{lemG3-3.1} claims that $H$ has a pair of
generators $(p,q),(-q,p-q)$, where $p^2-pq+q^2=n$. Suppose
$\ell(\nu)=\nu$ holds for some coset $\nu \in \ZZ^2/H$. Let
$(a,b)\in \ZZ^2$ be a representative of coset $\nu$. Then
$\nu-\ell(\nu)=(a+b,-a+2b)\in i(p,q)+j(-q,p-q)$. That is
$(a,b)=i(\frac{2p-q}{3},\frac{p+q}{3})+j(\frac{-p-q}{3},\frac{p-2q}{3})$
for some integer $i,j$. Then modulo $\langle (p,q),(-q,p) \rangle$
there are only three different choices for pairs $(a,b)$
corresponding to $i=j=0$, $i=j=1$ and $i=j=2$. The first pair is
always integer, the latter two are integer if and only if $p+q\equiv
0 \mod 3$. Also, $p+q\equiv 0 \mod 3$ if and only if $3\mid
p^2-pq+q^2=n$. That is, for a fixed $H$ there is one choice of $\nu$
if $3\nmid n$ and three choices if $3\mid n$. By \Cref{lemG3-3.1},
the number of possible subgroups $H$ is $\theta(n)$. So
$R(n)=\theta(n)$ if $3 \nmid n$ and $R(n)=3\theta(n)$ if $3 \mid n$.
Finally note $\theta(\frac{n}{3})=\theta(n)$ if $3 \mid n$ and
$\theta(\frac{n}{3})=0$ otherwise. Then we have the required
$R(n)=\theta(n)+2\theta(\frac{n}{3})$.
\end{proof}
\begin{remark}\label{number of thirds-remark}
Similar to \Cref{number of halfes-remark} get
$$
T(n)=\sum_{\ell(H)=H< \ZZ^2,\,|\ZZ^2/H|=n}|\{\nu| \nu\in
\ZZ^2/H,\,\ell(\nu)=\nu\}|=\sum_{\ell(H)=H< \ZZ^2,\,|\ZZ^2/H|=n}
|\ZZ^2/\langle (1,-1), (1,2), H\rangle|.
$$
\end{remark}

\section{On the covering of $\mathcal{G}_{3}$}
\subsection{The structure of the group $\pi_{1}(\mathcal{G}_{3})$ }
The following proposition provides the canonical form of an element
in $\pi_{1}(\mathcal{G}_{3})=\langle x, y, z: xyx^{-1}y^{-1}=1,
   zxz^{-1}= y, zyz^{-1}=(xy)^{-1}\rangle$. The proof is similar to the proof of
Proposition 2 in \cite{We2}.

\begin{proposition}\label{propG3-1}
\begin{itemize}
\item[(i)] Each element of $\pi_{1}(\mathcal{G}_{3})$ can be represented in the canonical form $x^ay^bz^c$
for some integer $a,b,c$.
\item[(ii)] The product of two canonical forms is given by the
formula
\begin{equation}\label{multlawG3}
x^ay^bz^c \cdot x^{d}y^{e}z^{f}= \left\{
\begin{aligned}
x^{a+d}y^{b+e}z^{c+f} \quad \text{if} \quad c\equiv 0\mod3 \\
{x}^{a-e}{y}^{b+d-e}{z}^{c+f} \quad \text{if} \quad c\equiv 1\mod3 \\
{x}^{a-d+e}{y}^{b-d}{z}^{c+f} \quad \text{if} \quad c\equiv 2\mod3 \\
\end{aligned} \right.
\end{equation}
\item[(iii)] The canonical epimorphism $\phi_{\mathcal{G}3}: \pi_{1}(\mathcal{G}_{3}) \to
\pi_{1}(\mathcal{G}_{3})/\langle x,y\rangle \cong \ZZ$, given by the
formula $x^ay^bz^c \to c$ is well-defined.
\item[(iv)] The representation in the canonical form $g=x^ay^bz^c$ for each element $g \in \pi_{1}(\mathcal{G}_{3})$ is
unique.
\end{itemize}
\end{proposition}
Routinely follows from the definition of the group.

\smallskip
{\bf Notation.} By $\Gamma$ denote the subgroup of
$\pi_{1}(\mathcal{G}_{3})$ generated by $x,y$.
\smallskip

Our goal is to introduce some easy invariants, similar to those of
\Cref{number of sublattices}.


\begin{definition}\label{defG3-all}
Suppose all elements of $\pi_{1}(\mathcal{G}_{3})$ are represented
in the canonical form. Let $\Delta$ be a subgroup of finite index
$n$ in $\pi_{1}(\mathcal{G}_{3})$. Put $H(\Delta)=\Delta\bigcap
\Gamma$. By $a(\Delta)$ denote the minimal positive exponent of $z$
among all the elements of $\Delta$. Choose an element $Z_\Delta$
with such exponent of $z$, represented in the form
$Z_\Delta=hz^{a(\Delta)}$, where $h \in \Gamma$. By
$\nu(\Delta)=hH(\Delta)$ denote the coset in the coset decomposition
$\Gamma/H(\Delta)$.
\end{definition}

Note that the invariants $a(\Delta)$, $H(\Delta)$ and $\nu(\Delta)$
are well-defined. In particular the latter one does not depends on a
choice of $Z_\Delta$. Also
$a(\Delta)[\Gamma:H(\Delta)]=[\pi_{1}(\mathcal{G}_{2}):\Delta]$.

\begin{definition}\label{defG3-3-plet}
A 3-plet $(a,H,\nu)$ is called {\em $n$-essential} if the following
conditions holds:
\begin{itemize}
\item[(i)] $a$ is a positive divisor of $n$,
\item[(ii)] $H$ is a subgroup of index $n/a$ in $\Gamma$ with $H\lhd \pi_1(\mathcal{G}_3)$ if $3\nmid
a$,
\item[(iii)] $\nu$ is an element of $\Gamma/H$.
\end{itemize}
\end{definition}

The next proposition show that the introduced invariants are
sufficient to enumerate the subgroups of finite index.

\begin{proposition}\label{propG3-4}
There is a bijection between the set of $n$-essential 3-plets
$(a,H,\nu)$ and the set of subgroups $\Delta$ of index $n$ in
$\pi_1(\mathcal{G}_5)$, such that
$(a,H,\nu)=(a(\Delta),H(\Delta),\nu(\Delta))$, given by the
correspondence $\Delta \leftrightarrow
(a(\Delta),H(\Delta),\nu(\Delta))$. Moreover, $\Delta \cong \ZZ^3$
if $3\mid a(\Delta)$ and $\Delta\cong \pi_{1}(\mathcal{G}_{3})$
otherwise.
\end{proposition}

The next few lemmas are auxiliary statements needed for the proof of
\Cref{propG3-4}.

\begin{lemma}\label{lemG3-3.3}
If $3 \nmid a(\Delta)$ then $H(\Delta)\lhd \pi_1(\mathcal{G}_3)$.
\end{lemma}

\begin{proof}
Recall $Z_\Delta=hz^{a(\Delta)} \in \Delta$. Then
$H(\Delta)^{Z_\Delta}=H(\Delta)$. Since
$H(\Delta)^x=H(\Delta)^y=H(\Delta)^{z^3}=H(\Delta)$, the former fact
yields $H(\Delta)^g=H(\Delta),\; g\in \pi_1(\mathcal{G}_3)$.
\end{proof}

\begin{lemma}\label{lemG3-3.6}
For arbitrary $n$-essential 3-plet $(a,H,\nu)$ there exists a
subgroup $\Delta$ in the group $\pi_1(\mathcal{G}_3)$ such that
$(a,H,\nu)=(a(\Delta),H(\Delta),\nu(\Delta))$.
\end{lemma}

\begin{proof}
Take an $n$-essential 3-plet $(a,H,\nu)$. In case $3\nmid a$
consider the following construction
$$
\aligned & \Delta = \{hz^{(3l+1)a(\Delta)}H|l\in
\ZZ\}\bigcup\{hz^{a(\Delta)}hz^{-a(\Delta)}z^{(3l+2)a(\Delta)}H|l\in
\ZZ\}\bigcup \{z^{3la(\Delta)}H|l\in \ZZ\}.
\endaligned
$$
One can check that $\Delta$ is a subgroup of index $n$ in
$\pi_1(\mathcal{G}_3)$ and
$(a(\Delta),H(\Delta),\nu(\Delta))=(a,H,\nu)$.

Similarly, in case $3\mid a$ the set
$$
\Delta = \{h^lz^{la(\Delta)}H|l\in \ZZ\}
$$
is the required subgroup.
\end{proof}

\begin{proof}[\bf Proof of \Cref{propG3-4}]
 Consider the family of subgroups
$\Delta$ of index $n$ in $\pi_1(\mathcal{G}_3)$, and the family of
$n$-essential 3-plets. The definition of notions $a(\Delta)$,
$H(\Delta)$ and $\nu(\Delta)$ together with \Cref{lemG3-3.3} provide
the correspondence from subgroups to $n$-essential 3-plets. Since
the above invariants are well-defined, each subgroup $\Delta$
corresponds to only one 3-plet. By virtue of \Cref{lemG3-3.6} each
3-plet corresponds to some subgroup, also different 3-plets
correspond to different 3-plets. The bijection part is proven.

If $3 \mid a(\Delta)$ \Cref{lemG3-3.6} implies that $\Delta$ is a
subgroup of $\langle x,y,z^3\rangle$. Thus $\Delta$ is a subgroup of
finite index in $\ZZ^3$, hence $\Delta$ is isomorphic to $\ZZ^3$
itself.

Consider case $a(\Delta)\equiv 1 \mod 3$. Since $H(\Delta)$ is a
subgroup of finite index in $\langle x,y\rangle \cong \ZZ^2$, we
have $H(\Delta)\cong \ZZ^2$. By Lemmas \ref{lemG3-3.1} and
\ref{lemG3-3.3} there is a pair of elements, that generates
$H(\Delta)$, that have the form  form $(x^py^q, x^{-q}y^{p-q})$. Let
$h$ be an arbitrary element in the coset $\nu(\Delta)$. Put
$X=x^py^q$, $Y=x^{-q}y^{p-q}$ and $Z=hz^{a(\Delta)}$. Direct
verification shows that the relations $XYX^{-1}Y^{-1}=1$, $ZXZ^{-1}=
Y$, and $ZYZ^{-1}=(XY)^{-1}$ holds. Further we call this relations
{\em the proper relations of the subgroup $\Delta$}. Thus the map $x
\to X, \,\, y \to Y,\,\, z \to Z$ can be extended to an epimorphism
$\pi_1(\mathcal{G}_3) \to \Delta $. To prove that this epimorphism
is really an isomorphism we need to show that each relation in
$\Delta$ is a corollary of proper relations. We call a relation,
that is not a corollary of proper relations an {\em improper
relation}.

Assume the contrary, i.e. there are some improper relations in
$\Delta$. Since in $\Delta$ the proper relations holds, each element
can be represented in the canonical form, given by \Cref{propG3-1}
in terms of $X,Y,Z$, by using just the proper relations. That is
each element $g$ can be represented as
$$
g=X^{a}Y^{b}Z^{c}.
$$
If there is an improper relation then there is an equality
\begin{equation}\label{absurdum1-G3}
X^{a}Y^{b}Z^{c}=X^{a'}Y^{b'}Z^{c'},
\end{equation}
where at least one of the inequalities $a\neq a'$, $b\neq b'$,
$c\neq c'$ holds. Applying $\phi_{\mathcal{G}3}$ to both parts we
get $ca(\Delta)=c'a(\Delta)$, thus $c=c'$. Then
$U^{a}V^{b}=U^{a'}V^{b'}$, that means
\begin{equation}\label{absurdum2-G3}
\left\{
\begin{aligned}
ap-bq= a'p-b'q\\
aq+b(p-q)= a'q+b'(p-q)\\
\end{aligned} \right.
\end{equation}

Since $\det\begin{pmatrix} p & -q \\
q & p-q\end{pmatrix}=p^2-pq+q^2=\frac{n}{a(\Delta)}\neq0$, the
system \ref{absurdum2-G3} implies $a=a'$ and $b=b'$, this
contradiction proves that $\pi_1(\mathcal{G}_3) \cong \Delta$.

The proof in case $a(\Delta)\equiv 2 \mod 3$ is similar with the
only difference we take $Y=x^py^q$ and $X=x^{-q}y^{p-q}$.
\end{proof}

\subsection{The proof of \Cref{th-1-tricosm}}

Proceed to the proof of \Cref{th-1-tricosm}. \Cref{propG3-4} claims
that each subgroup $\Delta$ of finite index $n$ is isomorphic to
 $\pi_1(\mathcal{G}_{3})$ or $\ZZ^3$, depending upon whether
$a(\Delta)$ is a multiple of 3. Consider these two cases separately.

{\bf Case (i).} Let $\Delta$ be a subgroup of
$\pi_1(\mathcal{G}_{3})$ isomorphic to $\pi_1(\mathcal{G}_{3})$. To
find the number of such subgroups, by \Cref{propG3-4} we need to
calculate the cardinality of the set of $n$-essential 3-plets with
$3 \nmid a$.

For each $3 \nmid a$ there are $\theta(\frac{n}{a})$ subgroups $H$
in $\Gamma$ such that $\big|\Gamma:H\big|=\frac{n}{a}$ and $H\lhd
\pi_1(\mathcal{G}_3) $. Also there are $\frac{n}{a}$ different
choices of a coset $\nu$. Thus, for each $3 \nmid a$ the number of
$n$-essential 3-plets is $\frac{n}{a}\theta(\frac{n}{a})$. So, the
total number of subgroups is given by
$$
s_{\pi_1(\mathcal{G}_{3}), \pi_1(\mathcal{G}_{3})}(n) =  \sum_{a
\mid n,\, 3\nmid a}  \frac{n}{a}\theta(\frac{n}{a})=\sum_{a \mid n}
\frac{n}{a}\theta(\frac{n}{a})-\sum_{3a \mid n}
\frac{n}{3a}\theta(\frac{n}{3a})=\sum_{k \mid n}k\theta(k)-\sum_{k
\mid \frac{n}{3}}k\theta(k).
$$

{\bf Case (ii).} Similarly to the previous case, we get the formula
$$
s_{\ZZ^3, \pi_1(\mathcal{G}_{3})}(n) =  \sum_{3a \mid n}
\frac{n}{3a}\sigma_1(\frac{n}{3a})=\omega(\frac{n}{3}).
$$

\subsection{The proof of \Cref{th-2-tricosm}}
\subsubsection{Overall scheme of the proof}\label{scheme}
The proof of both cases follows the general scheme that we describe
first. Recall that a subgroup $G$ of finite index in
$\pi_{1}(\mathcal{G}_{3})$ has one of the following isomorphism
types: $\ZZ^3$ or $\pi_{1}(\mathcal{G}_{3})$. We use the standard
notation $[g_1,g_2]=g_1g_2g_1^{-1}g_2^{-1}, g_1,g_2 \in
\pi_{1}(\mathcal{G}_{3})$. Also, given subgroups $G_1,G_2 \leqslant
\pi_{1}(\mathcal{G}_{3})$ by $[G_1,G_2]$ we denote the subgroup,
generated by the elements $[g_1,g_2],\, g_1\in G_1, g_2\in G_2$. Fix
an isomorphism type of a subgroup. Further $\Delta$ will always
denote a subgroup of this isomorphism type of index $n$ in
$\pi_{1}(\mathcal{G}_{3})$. In each case we point a normal subgroup
of finite index $\Lambda \unlhd \pi_{1}(\mathcal{G}_{3})$ such that
two conditions are met:
\begin{enumerate}
\item[$1^\circ$] for any $\lambda \in \Lambda$ and any $\Delta$ holds  $Ad_\lambda(H(\Delta))=H(\Delta)$,
\item[$2^\circ$]$[\Lambda,Z(\Delta)]=[H(\Lambda),Z(\Delta)]$, where $Z(\Delta)$ is given by \Cref{defG3-all}.
\end{enumerate}

Call {\em an intermediate conjugacy class} $\Delta^\Lambda$ of
$\Delta$ the set of subgroups $\Delta^\lambda,\; \lambda\in\Lambda$.
Denote the number of such classes by
$c_{G,\pi_{1}(\mathcal{G}_{3})}^\Lambda$, where $G$ isomorphic to
one of $\ZZ^3$ or $\pi_{1}(\mathcal{G}_{3})$.

We propose an algorithm to  uniformly calculate
$c_{G,\pi_{1}(\mathcal{G}_{3})}^\Lambda$. Given $\Delta$, a subgroup
$\Delta'\in \Delta^\Lambda$ have the following invariants:
$a(\Delta')=a(\Delta)$, $H(\Delta')=H(\Delta)$ and
$\nu(\Delta')\in\nu(\Delta)[\Lambda,Z(\Delta)]$. Keep in mind that
$[\Lambda,Z(\Delta)] \leqslant \Gamma$, since $\Gamma$ is normal in
$\pi_{1}(\mathcal{G}_{3})$ and $\pi_{1}(\mathcal{G}_{3})/\Gamma$ is
abelian. Thus for a fixed pair $(a,H)$ there are
$|\Gamma:\langle[\Lambda,Z(\Delta)],H\rangle|$ partial conjugacy
classes $\Delta^\Lambda$, each corresponding to the pair $(a,H)$.
This let us to enumerate partial conjugacy classes.

The factor-group $\pi_{1}(\mathcal{G}_{3})/\Lambda$ acts by
conjugation on partial conjugacy classes. An orbit of partial
conjugacy classes form a conjugacy class, thus we use the Burnside's
lemma to calculate the number of conjugacy classes. To do this, we
introduce one more definition.


\begin{definition} Given $u \in \pi_{1}(\mathcal{G}_{3})/\Lambda$, by
$B(u)$ denote the number of partial conjugacy classes, preserved by
the conjugation with $u$: $B(u)=|\{\Delta^\Lambda|
\big(\Delta^\Lambda\big)^u=\Delta^\Lambda\}|$. In particular, $B(1)$
is the number of partial conjugacy classes.
\end{definition}

Now we are done with the general scheme and proceed to its
realization in specific cases.

\subsubsection{Case (i)}
Put $\Lambda=\pi_{1}(\mathcal{G}_{3})$. \Cref{propG3-4} claims
$H(\Delta)\vartriangleleft\pi_{1}(\mathcal{G}_{3})$ in case
$\Delta\cong\pi_{1}(\mathcal{G}_{3})$, thus ($1^\circ$) holds.
Recall that $Z(\Delta)=hz^{a(\Delta)}$. Direct calculation through
(\ref{multlawG3}) shows that $[\Lambda,Z(\Delta)]=\langle
xy^{-1},xy^2 \rangle=[H(\Lambda),Z(\Delta)]$.

This means, firstly, that ($2^\circ$) holds, and secondly that
$|\Gamma:\langle[\Lambda,Z(\Delta)],H\rangle|$ equals 1 if $3\nmid
n$ and equals 3 if $3\mid n$. For a fixed $a(\Delta)$ the number of
subgroups $H(\Delta)$ is $\theta(\frac{n}{a(\Delta)})$. Keep in mind
that $\theta(\frac{k}{3})=\theta(k)$ if $\frac{k}{3}$ is integer,
and $\theta(\frac{k}{3})=0$ otherwise. So the function $k \mapsto
\left\{\begin{aligned}
\theta(k)  \;\text{if} \; 3 \nmid k\\
3\theta(k)  \;\text{if} \; 3 \mid k\\
\end{aligned}
\right.$ is given by $\theta(k)+2\theta{\frac{k}{3}}$. Applying this
and summing achieved number of pairs $(H,\nu)$ over all values of
$a$ one gets:
$$
\begin{aligned} &
c_{\pi_{1}(\mathcal{G}_{3}),\pi_{1}(\mathcal{G}_{3})}^\Lambda(n) =
\sum_{a\mid n,\,3\nmid a}
\theta(\frac{n}{a})+2\theta(\frac{n}{3a})=\sum_{a\mid
n}\theta(\frac{n}{a})+\sum_{a\mid
\frac{n}{3}}\theta(\frac{n}{3a})-2\sum_{a\mid
\frac{n}{9}}\theta(\frac{n}{9a})=\\& \sum_{k\mid
n}\theta(k)+\sum_{k\mid \frac{n}{3}}\theta(k)-2\sum_{k\mid
\frac{n}{9}}\theta(k).
\end{aligned}
$$
Since $\Lambda=\pi_{1}(\mathcal{G}_{3})$, the Burnside's lemma is
not needed to conclude
$c_{\pi_{1}(\mathcal{G}_{3}),\pi_{1}(\mathcal{G}_{3})}(n)
=c_{\pi_{1}(\mathcal{G}_{3}),\pi_{1}(\mathcal{G}_{3})}^\Lambda(n)$.

\subsubsection{Case (ii)}
Put $\Lambda=\langle x,y,z^3\rangle$. Since $\Lambda \cong\ZZ^3$ and
$\Delta\leqslant\Lambda$ for any $\Delta$, conditions ($1^\circ$)
and ($2^\circ$) hold. Also each partial conjugacy class consists of
one subgroup, i.e.
$c_{\ZZ^3,\pi_{1}(\mathcal{G}_{3})}^\Lambda(n)=s_{\ZZ^3,\pi_{1}(\mathcal{G}_{3})}(n)=\omega(\frac{n}{3})$.

The factor $\pi_{1}(\mathcal{G}_{3})/\Lambda$ consists of three
elements, represented by $1$, $z$, and $z^2$ respectively.

The condition $\Delta^{z}=\Delta$ is equivalent to
$(H(\Delta))^{z}=H(\Delta)$ and $(\nu(\Delta))^{z}=\nu(\Delta)$ met
simultaneously. \Cref{number of thirds} provides the number of such
pairs $(H,\nu)$ for a given value of $a$. Summing over all the
possible values (recall that $3 \mid a$) one gets
$B(z)=B(z^2)=\sum_{k\mid\frac{n}{3}}\theta(k)+2\sum_{k\mid\frac{n}{9}}\theta(k)$.

By making use of Burnside's lemma obtain
$$
c_{\ZZ^3,\pi_{1}(\mathcal{G}_{3})}(n)=\frac{1}{3}\Big(\omega(\frac{n}{3})
+2\sum_{k\mid\frac{n}{3}}\theta(k)+4\sum_{k\mid\frac{n}{9}}\theta(k)\Big).
$$

\section{On the coverings of $\mathcal{G}_{5}$}\label{partG5}
\subsection{The structure of the group $\pi_{1}(\mathcal{G}_{5})$ }
Recall that the fundamental group of  $\mathcal{G}_{5}$ is given by:
$\pi_{1}(\mathcal{G}_{5})=\langle \tilde{x}, \tilde{y}, \tilde{z}:
\tilde{x}\tilde{y}\tilde{x}^{-1}\tilde{y}^{-1}=1,
   \tilde{z}\tilde{x}\tilde{z}^{-1}= \tilde{x}\tilde{y}, \tilde{z}\tilde{y}\tilde{z}^{-1}=\tilde{x}^{-1}
 \rangle$.
The following proposition provides the canonical form of an element
in $\pi_{1}(\mathcal{G}_{5})$.

\begin{proposition}\label{propG5-1}
\begin{itemize}
\item[(i)] Each element of $\pi_{1}(\mathcal{G}_{5})$ can be represented in the canonical form $\tilde{x}^a\tilde{y}^b\tilde{z}^c$
for some integer $a,b,c$.
\item[(ii)] The product of two canonical forms is given by the
formula
\begin{equation}\label{multlawG5}
\tilde{x}^a\tilde{y}^b\tilde{z}^c \cdot
\tilde{x}^{d}\tilde{y}^{e}\tilde{z}^{f}= \left\{
\begin{aligned}
\tilde{x}^{a+d}\tilde{y}^{b+e}\tilde{z}^{c+f} \quad \text{if} \quad c\equiv 0\mod6 \\
\tilde{x}^{a+d-e}\tilde{y}^{b+d}\tilde{z}^{c+f} \quad \text{if} \quad c\equiv 1\mod6 \\
\tilde{x}^{a-e}\tilde{y}^{b+d-e}\tilde{z}^{c+f} \quad \text{if} \quad c\equiv 2\mod6 \\
\tilde{x}^{a-d}\tilde{y}^{b-e}\tilde{z}^{c+f} \quad \text{if} \quad c\equiv 3\mod6 \\
\tilde{x}^{a-d+e}\tilde{y}^{b-d}\tilde{z}^{c+f} \quad \text{if} \quad c\equiv 4\mod6 \\
\tilde{x}^{a+e}\tilde{y}^{b-d+e}\tilde{z}^{c+f} \quad \text{if} \quad c\equiv 5\mod6 \\
\end{aligned} \right.
\end{equation}
\item[(iii)] The canonical epimorphism $\phi_{\mathcal{G}5}: \pi_{1}(\mathcal{G}_{5}) \to
\pi_{1}(\mathcal{G}_{5})/\langle \tilde{x},\tilde{y}\rangle \cong
\ZZ$, given by the formula $\tilde{x}^a\tilde{y}^b\tilde{z}^c \to c$
is well-defined.
\item[(iv)] The representation in the canonical form $g=\tilde{x}^a\tilde{y}^b\tilde{z}^c$ for
each element $g \in \pi_{1}(\mathcal{G}_{5})$ is unique.
\end{itemize}
\end{proposition}

Routinely follows from the definition of the group.

\smallskip
{\bf Notation.} Let $\Gamma=\langle \tilde{x},\tilde{y}\rangle$ be
the subgroup of $\pi_{1}(\mathcal{G}_{5})$ generated by
$\tilde{x},\tilde{y}$.
\smallskip

Our goal is to introduce some easy invariants, similar to those in
\Cref{number of sublattices}. That will let us to enumerate the
subgroups.

\begin{definition}\label{defG5-all}
Suppose all elements of $\pi_{1}(\mathcal{G}_{5})$ are represented
in the canonical form. Let $\Delta$ be a subgroup of finite index
$n$ in $\pi_{1}(\mathcal{G}_{5})$. Put $H(\Delta)=\Delta\bigcap
\Gamma$. By $a(\Delta)$ denote the minimal positive exponent of
$\tilde{z}$ among all the elements
$\tilde{x}^a\tilde{y}^b\tilde{z}^c\in\Delta$. Choose an element
$Z(\Delta)$ with such exponent of $\tilde{z}$, represented in the
form $Z(\Delta)=h\tilde{z}^{a(\Delta)}$, where $h \in \Gamma$. By
$\nu(\Delta)=hH(\Delta)$ denote the coset in $\Gamma/H(\Delta)$
containing $h$.
\end{definition}

\begin{definition}\label{defG5-3-plet}
A 3-plet $(a,H,\nu)$ is called {\em $n$-essential} if the following
conditions holds:
\begin{itemize}
\item[(i)] $a$ is a positive divisor of $n$,
\item[(ii)] $H$ is a subgroup of index $n/a$ in $\Gamma$ with $H\lhd \pi_1(\mathcal{G}_5)$ if $3\nmid
a$,
\item[(iii)] $\nu$ is an element of $\Gamma/H$.
\end{itemize}
\end{definition}

\begin{proposition}\label{propG5-4}
There is a bijection between the set of $n$-essential 3-plets
$(a,H,\nu)$ and the set of subgroups $\Delta$ of index $n$ in
$\pi_1(\mathcal{G}_5)$, given by the correspondence $\Delta
\leftrightarrow (a,H,\nu)=(a(\Delta),H(\Delta),\nu(\Delta))$.
Moreover, $\Delta \cong \ZZ^3$ if $(a,6)=6$, $\Delta \cong
\pi_1(\mathcal{G}_2)$ if $(a,6)=3$, $\Delta \cong
\pi_1(\mathcal{G}_3)$ if $(a,6)=2$ and $\Delta\cong
\pi_{1}(\mathcal{G}_{5})$ if $(a,6)=1$.
\end{proposition}

\begin{proof} The proof of \Cref{propG5-4} is similar to the proof
of \Cref{propG3-4}.
\end{proof}

\subsection{The proof of \Cref{th-1-hexacosm}}
In case (i) the argument similar to the proof of \Cref{th-1-tricosm}
leads to the formula:
$$
s_{\pi_1(\mathcal{G}_{5}), \pi_1(\mathcal{G}_{5})}(n) =  \sum_{a
\mid n,\;(a,6)=1}\frac{n}{a}\theta(\frac{n}{a})= \sum_{k \mid
n,\;(\frac{n}{k},6)=1}k\theta(k).
$$
The last equality is obtained by applying the inclusion–exclusion
principle. The cases (ii), (iii) and (iv) can be treated in the
similar way.

\subsection{The proof of \Cref{th-2-hexacosm}}
The proof uses the overall scheme form \cref{scheme}, so we just
proceed to to its realization in specific cases.

\subsubsection{Case (i)}
Put $\Lambda=\pi_{1}(\mathcal{G}_{5})$. \Cref{propG5-4} claims
$H(\Delta)\vartriangleleft\pi_{1}(\mathcal{G}_{5})$ in case
$\Delta\cong\pi_{1}(\mathcal{G}_{5})$, thus ($1^\circ$) holds.
Direct calculation using (\ref{multlawG5}) in case $a(\Delta)\equiv
1 \mod 6$ gives $[\tilde{x},Z(\Delta)]=\tilde{y}^{-1}$ and
$[\tilde{y},Z(\Delta)]=\tilde{x}\tilde{y}$. In case $c(\Delta)\equiv
5 \mod 6$ we respectively get
$[\tilde{x},Z(\Delta)]=\tilde{x}\tilde{y}$ and
$[\tilde{y},Z(\Delta)]=\tilde{x}^{-1}$. So in both cases
$[\Gamma,Z(\Delta)]=\Gamma$. Firstly this means that ($2^\circ$)
holds. Secondly, conjugacy classes of subgroups $\Delta$ are in
one-to-one correspondence with pairs $(a,H)$. Summing the number of
choices of $H$ over all the possible values of $a$ get
$$
c_{\pi_{1}(\mathcal{G}_{5}),\pi_{1}(\mathcal{G}_{5})}(n) =
\sum_{a\mid n,\,(6,a)=1} \theta(\frac{n}{a})=\sum_{k\mid n}
\theta(k) - \sum_{k\mid \frac{n}{2}} \theta(k) - \sum_{k\mid
\frac{n}{3}} \theta(k) + \sum_{k\mid \frac{n}{6}} \theta(k).
$$


\subsubsection{Case (ii)}
Put $\Lambda=\langle \tilde{x},\tilde{y},\tilde{z}^2\rangle$.
\Cref{propG5-4} claims
$H(\Delta)\vartriangleleft\pi_{1}(\mathcal{G}_{5})$ in case
$\Delta\cong\pi_{1}(\mathcal{G}_{5})$, thus ($1^\circ$) holds.
Recall that $Z(\Delta)=h\tilde{z}^{a(\Delta)}$. Direct calculation
through (\ref{multlawG5}) shows that $[\Lambda,Z(\Delta)]=\langle
\tilde{x}\tilde{y}^{-1},\tilde{x}\tilde{y}^2
\rangle=[H(\Lambda),Z(\Delta)]$.

This means, firstly, that ($2^\circ$) holds, and secondly that
$|\Gamma:\langle[\Lambda,Z(\Delta)],H\rangle|$ equals 1 if $3\nmid
n$ and equals 3 if $3\mid n$.

The factor $\pi_{1}(\mathcal{G}_{5})/\Lambda$ consists of two
elements, represented by $1$ and $\tilde{x}^3$ respectively. The
conjugation with these elements preserves $(a,H)$, thus in case
$3\nmid n$ the partial conjugacy classes coincide with conjugacy
classes, and there is only one conjugacy class for a fixed pair
$(a,H)$. In case $3\mid n$ for a fixed pair $(a,H)$ there are 3
partial conjugacy classes: namely $\Delta_0^\Lambda$,
$\Delta_1^\Lambda$ and $\Delta_2^\Lambda$, where
$\Delta_0\leftrightarrow (a,H,1)$, $\Delta_1\leftrightarrow
(a,H,\tilde{y})$ and $\Delta_2\leftrightarrow (a,H,\tilde{y}^2)$.
Note that the conjugation with $\tilde{x}^3$ swaps the partial
conjugacy classes $\Delta_1^\Lambda$ and $\Delta_2^\Lambda$. Thus
for a fixed pair $(a,H)$ there are two conjugacy classes:
$\Delta_0^{\pi_1(\mathcal{G}_5)}=\Delta_0^{\Lambda}$ and
$\Delta_1^{\pi_1(\mathcal{G}_5)}=\Delta_1^{\Lambda}\bigcup\Delta_2^{\Lambda}$.

Keep in mind that $\theta(\frac{n}{3})=\theta(n)$ if $\frac{n}{3}$
is integer, and $\theta(\frac{n}{3})=0$ otherwise. Applying this and
summing achieved number of conjugacy classes over all values of $a$
one gets:
$$
\begin{aligned} &
c_{\pi_{1}(\mathcal{G}_{3}),\pi_{1}(\mathcal{G}_{5})}(n) =
\sum_{a\mid \frac{n}{2},\,a\nmid \frac{n}{6},\,3\nmid a}
\theta(\frac{n}{2a}) +2\sum_{a\mid \frac{n}{6},\,3\nmid a}
\theta(\frac{n}{6a})=\sum_{a\mid \frac{n}{2},\,3\nmid a}
\theta(\frac{n}{2a}) +\sum_{a\mid \frac{n}{6},\,3\nmid a}
\theta(\frac{n}{6a}) =\\&\sum_{a \mid
\frac{n}{2}}\theta(\frac{n}{2}) -\sum_{a \mid
\frac{n}{18}}\theta(\frac{n}{18})=
\sum_{k\mid\frac{n}{2}}\theta(k)-\sum_{k\mid\frac{n}{18}}\theta(k).
\end{aligned}
$$

\subsubsection{Case (iii)}
Put $\Lambda=\langle \tilde{x},\tilde{y},\tilde{z}^3\rangle$.
$Ad_{\tilde{x}}$ and $Ad_{\tilde{y}}$ are the identity
transformation on $\Gamma$. $Ad_{\tilde{z}^3}$ is given by $g \to
g^{-1},\; g\in \Gamma$. That is $Ad_{\tilde{x}}$, $Ad_{\tilde{y}}$
and  $Ad_{\tilde{z}^3}$ preserves $H(\Delta)$, i.e. ($1^\circ$)
holds.

Further $[H(\Lambda),Z(\Delta)]=\langle \tilde{x}^2,\tilde{y}^2
\rangle$. Recall the notation $Z(\Delta)=h\tilde{z}^{a(\Delta)}$.
Then $[\tilde{z}^3,Z(\Delta)]=h^{-2}\in \langle
\tilde{x}^2,\tilde{y}^2 \rangle$, so ($2^\circ$) holds.

Applying \Cref{number of halfes-remark} and \Cref{number of halfes}
and summing over all possible values of $a$ find
$$
c_{\pi_{1}(\mathcal{G}_{2}),\pi_{1}(\mathcal{G}_{5})}^\Lambda(n)=\sum_{\substack{a
\mid n,\\3\mid a, 6\nmid a}
}\Big(\sigma_1(\frac{n}{a})+3\sigma_1(\frac{n}{2a})\Big)=\sigma_2(\frac{n}{3})+2\sigma_2(\frac{n}{6})-3\sigma_2(\frac{n}{12}).
$$

The factor $\pi_{1}(\mathcal{G}_{5})/\Lambda$ consists of three
elements, represented by $1$, $\tilde{x}^2$ and $\tilde{x}^4$
respectively. Obviously the numbers of partial conjugacy classes,
preserved by $Ad_{\tilde{x}^2}$ and $Ad_{\tilde{x}^4}$ coincide:
$B(\tilde{x}^2)=B(\tilde{x}^4)$. To find them note the following. A
partial conjugacy class $\Delta^\Lambda$ is preserved by
$Ad_{\tilde{x}^2}$ iff the hollowing conditions are met
simultaneously:
\begin{itemize}
\item[(*)]$\big(H(\Delta)\big)^{\tilde{z}^2}=H(\Delta)$
\item[(**)]$\big(\nu(\Delta)\big)^{\tilde{z}^2}$ must belong to the same
conjugacy class in $\Gamma/\langle \tilde{x}^2,\tilde{y}^2,
H\rangle$ as $\nu(\Delta)$.
\end{itemize}

By \Cref{lemG3-3.1} the condition (*) implies that $H(\Delta)$ have
a pair of generators
$(\tilde{x}^p\tilde{y}^q,\tilde{x}^{-q}\tilde{y}^{p-q})$. The matrix
$\begin{pmatrix}   p & q \\ -q & p-q \end{pmatrix}$ modulo 2 have
the rank 0 or 2, never 1. So $\Gamma/\langle
\tilde{x}^2,\tilde{y}^2,
\tilde{x}^p\tilde{y}^q,\tilde{x}^{-q}\tilde{y}^{p-q}\rangle$ is
either trivial or isomorphic to $\ZZ_2^2$. In the first case the
condition (**) holds for the sole element, in the second case the
condition (**) holds only for the coset of 0, and does not holds for
three other cosets.

So the conjugacy classes $\Delta^{\Lambda}$ with
$\Delta^{\tilde{z}^2}\in \Delta^{\Lambda}$ are in one-to-one
correspondence with the normal subgroups $H(\Delta)\lhd
\pi_{1}(\mathcal{G}_{5})$. Applying \Cref{lemG3-3.1} and summing
over all the possible values of $a$ yields
$B(\tilde{x}^2)=B(\tilde{x}^4)=\sum_{k\mid
\frac{n}{3}}\theta(k)-\sum_{k\mid \frac{n}{6}}\theta(k)$.
Substituting to Burnside's lemma obtain
$$
c_{\pi_{1}(\mathcal{G}_{2}),\pi_{1}(\mathcal{G}_{5})}(n)=\frac{1}{3}\Big(\sigma_2(\frac{n}{3})+2\sigma_2(\frac{n}{6})-3\sigma_2(\frac{n}{12})+2\sum_{k\mid
\frac{n}{3}}\theta(k)-2\sum_{k\mid \frac{n}{6}}\theta(k)\Big).
$$

\subsubsection{Case (iv)}
Put $\Lambda=\langle \tilde{x},\tilde{y},\tilde{z}^6\rangle$. Since
$\Lambda \cong\ZZ^3$ and $\Delta\leqslant\Lambda$ for any $\Delta$,
conditions ($1^\circ$) and ($2^\circ$) hold. Also each partial
conjugacy class consists of one subgroup, i.e.
$c_{\ZZ^3,\pi_{1}(\mathcal{G}_{5})}(n)^\lambda=s_{\ZZ^3,\pi_{1}(\mathcal{G}_{5})}(n)=\omega(\frac{n}{6})$.

The factor $\pi_{1}(\mathcal{G}_{5})/\Lambda$ consists of six
elements, represented by $1$, $\tilde{x}$, $\tilde{x}^2$,
$\tilde{x}^3$, $\tilde{x}^4$ and $\tilde{x}^5$ respectively.

The condition $\Delta^{\tilde{z}^3}=\Delta$ is equivalent to
$2\nu(\Delta)=0$. \Cref{number of halfes} provides the number of
such pairs $(H,\nu)$ for a given value of $a$, summing over all
possible values (recall that $6 \mid a$) one gets
$B(\tilde{z}^3)=\sigma_2(\frac{n}{6})+3\sigma_2(\frac{n}{12})$.

The condition $\Delta^{\tilde{z}^2}=\Delta$ is equivalent to
$(H(\Delta))^{\tilde{z}^2}=H(\Delta)$ and
$(\nu(\Delta))^{\tilde{z}^2}=\nu(\Delta)$ met simultaneously.
\Cref{number of halfes} provides the number of such pairs $(H,\nu)$
for a given value of $a$. Summing over all the possible values
(recall that $6 \mid a$) one gets
$B(\tilde{z}^2)=B(\tilde{z}^4)=\sum_{k\mid\frac{n}{6}}\theta(k)+2\sum_{k\mid\frac{n}{18}}\theta(k)$.

Finally, $\Delta^{\tilde{z}}=\Delta$ implies
$(H(\Delta))^{\tilde{z}}=H(\Delta)$ and
$(\nu(\Delta))^{\tilde{z}^3}=(\nu(\Delta))^{\tilde{z}^2}=\nu(\Delta)$.
The latter two equalities imply $\nu(\Delta)=0$, so the unique
subgroup $\Delta$ correspond to a $H(\Delta)$. Summing over all the
possible values of $a$ yields
$B(\tilde{z})=B(\tilde{z}^5)=\sum_{k\mid\frac{n}{6}}\theta(k)$.

Substituting to Burnside's lemma obtain
$$
c_{\ZZ^3,\pi_{1}(\mathcal{G}_{5})}(n)=\frac{1}{6}\Big(\omega(\frac{n}{6})+\sigma_2(\frac{n}{6})+3\sigma_2(\frac{n}{12})
+4\sum_{k\mid\frac{n}{6}}\theta(k)+4\sum_{k\mid\frac{n}{18}}\theta(k)\Big).
$$

\section{Appendix}
Given a sequence $\{f(n)\}_{n=1}^\infty$, a formal power series
$$
\widehat{f}(s)=\sum_{n=1}^\infty\frac{f(n)}{n^s}
$$
is called Dirichlet generating function for $\{f(n)\}_{n=1}^\infty$,
see for example, \cite{Wolfram}. For the way to reconstruct the
sequence $f(n)$ by $\widehat{f}(s)$ see Perron's formula (for
example \cite{Wolfram2}).

Here we present the Dirichlet generating functions for the
calculated sequences $s_{H,G}(n)$ and $c_{H,G}(n)$. Since theorems
1--4 provides he explicit formulas, the remaining can is done by
direct calculations which we omit here.

{\bf Notations.} By $\zeta(s)$ we denote the Riemann zeta function.
Define sequence $\{\chi(n)\}_{n=1}^\infty$ by
$\chi(n)=\frac{1}{\sqrt{3}}\Big((-\frac12+\frac{\sqrt{3}}{2}i)^n-(-\frac12-\frac{\sqrt{3}}{2}i)^n\Big)$
or equivalently $\chi(n)=\left\{\begin{aligned}
1  \;\text{if}\; n\equiv 1 \mod 3\\
-1  \;\text{if}\; n\equiv 2 \mod 3\\
\end{aligned}
\right. $. For the sake of brevity denote
$\vartheta(s)=\widehat{\chi}(s)$. Note that $\vartheta(s)$ is the
Dirichlet L-series for the multiplicative character $\chi(n)$.

\def\formi{$3^{-s}\zeta(s)\zeta(s-1)\zeta(s-2)$}
\def\formii{$6^{-s}\zeta(s)\zeta(s-1)\zeta(s-2)$}
\def\formiii{$3^{-s-1}\zeta(s)\Big(\zeta(s-1)\zeta(s-2)+
 {2(1+2\cdot3^{-s})}\zeta(s)\vartheta(s)\Big)$}
\def\formiv{$6^{-s-1}\zeta(s)\Big(\zeta(s-1)\zeta(s-2)+{(1+3\cdot2^{-s})\zeta(s)\zeta(s-1)}+4(1+3^{-s})\zeta(s)\vartheta(s)\Big)$}
\def\formv{$3^{-s}\big(1-2^{-s}\big)\zeta(s)\zeta(s-1)\zeta(s-2)$}
\def\formvi{$3^{-s}\big(1-2^{-s}\big)\zeta(s)^2\Big((1+3\cdot2^{-s})\zeta(s-1)+2\vartheta(s)\Big)$}
\def\formvii{$\big(1-3^{-s}\big)\zeta(s-1)^2\vartheta(s-1)$}
\def\formviii{$2^{-s}\big(1-3^{-s}\big)\zeta(s-1)^2\vartheta(s-1)$}
\def\formix{$\big(1-3^{-s}\big)\big(1+2\cdot3^{-s}\big)\zeta(s)^2\vartheta(s)$}
\def\formx{$2^{-s}\big(1-3^{-s}\big)\big(1+3^{-s}\big)\zeta(s)^2\vartheta(s)$}
\def\formxi{$\big(1-2^{-s}\big)\big(1-3^{-s}\big)\zeta(s-1)^2\vartheta(s-1)$}
\def\formxii{$\big(1-2^{-s}\big)\big(1-3^{-s}\big)\zeta(s)^2\vartheta(s)$}
$$
\begin{array}{|c|c|p{5.5cm}|p{8.5cm}|}\hline
 \multicolumn{2}{|c|}{$\backslashbox{H}{G}$} & $\mathcal{G}_{3}$ & $\mathcal{G}_{5}$ \\ \hline
 \multirow{2}*{$\ZZ^3$} & \widehat{s}_{H,G} & \formi & \formii \\ \cline{2-4}
 & \widehat{c}_{H,G} & \formiii & \formiv  \\ \hline
 \multirow{2}*{$\mathcal{G}_{2}$} & \widehat{s}_{H,G} &  & \formv \\ \cline{2-4}
 & \widehat{c}_{H,G} &  & \formvi  \\ \hline
  \multirow{2}*{$\mathcal{G}_{3}$} & \widehat{s}_{H,G} & \formvii & \formviii \\ \cline{2-4}
 & \widehat{c}_{H,G} & \formix & \formx  \\ \hline
 \multirow{2}*{$\mathcal{G}_{5}$} & \widehat{s}_{H,G} &  & \formxi \\ \cline{2-4}
 & \widehat{c}_{H,G} &  & \formxii  \\ \hline
\end{array}
$$

\end{document}